\documentclass[12pt,a4paper]{article}
\usepackage[centertags]{amsmath}
\usepackage{amsfonts}
\usepackage{amssymb}
\usepackage{amsthm}
\usepackage{color}
\usepackage{marginnote}
\usepackage{float}

\usepackage{bussproofs}

\usepackage{fullpage}
\usepackage{tikz}
\usepackage{verbatim}
\usetikzlibrary{arrows}

\newtheorem{cor}{Corollary}
\newtheorem{defi}{Definition}
\newtheorem{ex}{Example}
\newtheorem{rem}{Remark}
\newtheorem{prop}{Proposition}
\newtheorem{lem}{Lemma}

\newtheorem{teo}{Theorem}
\newtheorem{exer}{Exercise}
\newtheorem{no}{Notation}

\usepackage{color}

\begin{document}

\title{On a characterization of complete $o$-modular join semilattices}

\author{Rodolfo C. Ertola-Biraben}


\maketitle

\abstract Dedekind stated and proved the well-known fact that a lattice is modular if and only if it does not contain a pentagon as a sublattice. 
In this paper we consider a similar result in the literature for the case of certain class of modular join semilattices. 
We both simplify the original proof of the mentioned result and note that if the notion of characterization is undestood strictly as a biconditional, then the proof only holds for the complete algebras in the class.

\section{Introduction}

Since Dedekind (see \cite[Satz IX, p. 389]{D1900}), 
it is well known that a lattice is modular if and only if it does not contain a pentagon as a sublattice. 
In the sequel we will consider a similar result already appearing in \cite{R1990} for the case of $o$-modular join semilattices. 
In the rest of this short introduction we will provide both the basic notions and notations needed for the understanding of the next section.   

We begin introducing the following notations for the sets of lower and upper bounds of a subset of the universe of a given poset. 

\begin{no}
Let $(P,\leq)$ be a poset with $X, Y \subseteq P$. Then, 

(i) $L_XY = \{ x \in X: x \leq y$, for all $\in Y \}$ and

(ii) $U_XY = \{ x \in X: x \geq y$, for all $\in Y \}$. 
\end{no}

\noindent With respect to this definition, in the sequel $X$ is usually taken to be $P$ itself. 
Also, we will almost always take two element sets as the sets $Y$ in the given notations and will accordingly simply write $L_X(a,b)$ and $U_X(a,b)$.   
The only exception appears in the next definition, whose concept seems to have appeared for the first time in \cite{LR1988}. 

\begin{defi} \label{LRd}
A join semilattice $(S; \vee)$ is called \emph{$o$-modular} iff 

for all $a, b, c \in S$, if $c \leq a$, then $L_S(a, b \vee c) \subseteq L_SU_S(L_S(a,b),c)$. 
\end{defi}

\noindent One of the authors of the original paper \cite{LR1988} uses `$o$-modular' in \cite{R1992}. 

The following concepts appear also in \cite{R1990}. 
However, they are stated  for the general case of a poset. 
As in this paper we are dealing with join semilattices, the following simplification will be enough for our purpose. 
Also, note that we prefer to use `semi-strong' instead of `LU-subset' as the author does in \cite[p. 242]{R1990}. 

\begin{defi}
Let ${\bf{T}}$ with universe $T$ be a sub join semilattice of a join semilattice ${\bf{S}}$ with universe $S$. Then,  

(i) ${\bf{T}}$ is called \emph{semi-strong} if, for all $a, b \in T$, if $L_T (a, b) = \emptyset$, then $L_S(a,b) = \emptyset$.

(ii) ${\bf{T}}$ is called \emph{strong} if for all $a, b \in T$ it holds that $L_S(a,b) \subseteq L_T(a,b)$. 
\end{defi}

\noindent It should be clear that strongness implies semi-strongness. 

We will also use the join semilattices represented in Figure \ref{M2 and M4}. 
Note that $L(b,c)= \emptyset$ holds for the join semilattice on the left. 

\begin{figure} [ht]
\begin{center}

\begin{tikzpicture}

    \tikzstyle{every node}=[draw, circle, fill=white, minimum size=4pt, inner sep=0pt, label distance=1mm]


    \draw (1,.5)	node (b)	[label=right: $b$]	{};
    \draw (0,2)		node (ab)						{};
    \draw (-1,1)	node (ax)	[label=left: $c$]	{};
	\draw (-1,0)	node (a)						{};
      
        
	\draw (a)--(ax)--(ab)--(b);  
 

    \draw (5,.5)	node (b4)	[label=right: $b$]	{};
    \draw (4,2)		node (ab4)						{};
    \draw (3,1)		node (c4)	[label=left: $c$]	{};
	\draw (3,0)		node (a)	[label=left: $a$]	{};
    \draw (4,-1)	node (v4)						{};
      
        
	\draw (v4)--(a)--(c4)--(ab4)--(b4)--(v4);  
 
\end{tikzpicture}

\end{center}
\caption{\label{M2 and M4} Join semilattices $M_2$ and $M_4$}
\end{figure}
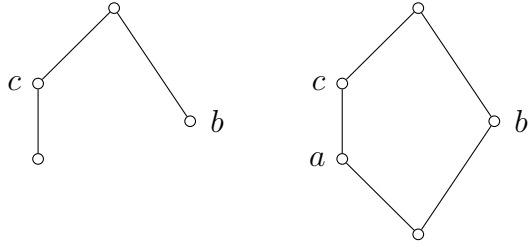

\section{A characterization} 

In spite of being straightforward, we provide a proof of the following fact. 

\begin{teo} \label{T1}
If a join semilattice contains either a semi-strong sub join semilattice isomorphic to $M_2$ or a strong sub join semilattice isomorphic to $M_4$, then it is not $o$-modular.  
\end{teo}

\begin{proof}
On the one hand, let $(S; \vee)$ be a join semilattice containing a semi-strong sub join semilattice $(T; \vee)$ isomorphic to $M_2$.
Then, labelling the elements as in the left structure in Figure \ref{M2andM4nm}, 
there exist $a, b, c, x$, and $y$ in $T$ (so also in $S$) such that $a < c$, $x \leq c$, $x \leq a \vee b$, $a \leq y$, and $x \nleq y$. 
Moreover, as $(T; \vee)$ is semi-strong and $L_T (b, c) = \emptyset$, it follows that $L_S (b, c) = \emptyset$.  
So, $(S; \vee)$ is not modular.

On the other hand, suppose $(S; \vee)$ is a join semilattice that contains a strong sub join semilattice $(T; \vee)$ isomorphic to $M_4$. 
Then, labelling the elements as in the right structure in Figure \ref{M2andM4nm}, there exist $a, b, c, x$, and $y$ in $T$ (so also in $S$) such that $a < c$, $x \leq c$, $x \leq a \vee b$, $a \leq y$, and $x \nleq y$. 
Moreover, take $z \in S$ such that $z \in L_S(b,c)$. 
As $(T; \vee)$ is strong, it follows that $z \in L_T(b,c)$. 
Now, it is clear in $T$ that $z \leq y$. 
So, $(S; \vee)$ is not modular. 
\end{proof}

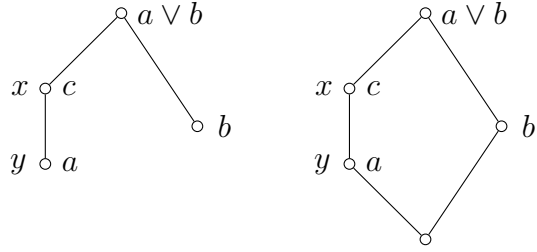
\begin{figure} [ht]
\begin{center}

\begin{tikzpicture}

    \tikzstyle{every node}=[draw, circle, fill=white, minimum size=4pt, inner sep=0pt, label distance=1mm]


    \draw (1,.5)	node (b)	[label=right: $b$]						{};
    \draw (0,2)		node (ab)	[label=right: $a \vee b$]				{};
    \draw (-1,1)	node (ax)	[label=left: $x$]	[label=right: $c$]	{};
	\draw (-1,0)	node (a)	[label=left: $y$]	[label=right: $a$]	{};
      
        
	\draw (a)--(ax)--(ab)--(b);  
 

    \draw (5,.5)	node (b4)	[label=right: $b$]	{};
    \draw (4,2)		node (ab4)	[label=right: $a \vee b$]	{};
    \draw (3,1)		node (c4)	[label=left: $x$]	[label=right: $c$] {};
	\draw (3,0)		node (av4)	[label=left: $y$]	[label=right: $a$]	{};
    \draw (4,-1)	node (v)		{};
      
        
	\draw (v)--(av4)--(c4)--(ab4)--(b4)--(v);  
 
\end{tikzpicture}

\end{center}
\caption{\label{M2andM4nm} Sub join semilattices appearing in the proof of Theorem \ref{T1}}
\end{figure}

The proof of the following theorem is simpler than the respective part of the proof of \cite[Theorem 2] {R1990}. 
This is because it is enough to consider two instead of four cases. 

\begin{teo} \label{T2}
If a join semilattice is not $o$-modular, then it contains a semi-strong sub join semilattice isomorphic to either $M_2$ or $M_4$. 
\end{teo}

\begin{proof}
Suppose that ${\bf S} = (S; \vee)$ is a non-$o$-modular join semilattice. 
Then, by Definition \ref{LRd} it follows that there exist $a, b, c, x, y \in S$ such that 
(i) $a < c$, 
(ii) $x \leq c$, 
(iii) $x \leq a \vee b$, 
(iv) for all $z$, if $z \leq b$ and $z \leq c$, then $z \leq y$, 
(v) $a \leq y$, and
(vi) $x \nleq y$.
Now, by (v) and (vi), it follows 

(vii) $x \nleq a$.  
Also, by (ii), (iii), (iv) and (vi), it follows that $a \nleq b$ and, 
using (iii) and (vii), it follows that $b \nleq a$. 
So, we have 

(viii) $a \parallel b$. 
By (iii), (iv), (v) and (vi), we also get  

(ix) $b \nleq c$. 

\noindent Now, by (vii), (viii), and (ix), we get  
$a < a \vee x < a \vee b$, $b < a \vee b$, $a \parallel b$, and $a \vee x \parallel b$. 
So, the set $T_2 = \{ a, b, a \vee x, a \vee b \}$ (see Figure \ref{T2andT4}) is a sub join semilattice of ${\bf S}$ isomorphic to $M_2$. 
If we further assume that $L_S(b,c) = \emptyset$, then $T_2$ is semi-strong. 

Let us now suppose that $L_S(b,c) \neq \emptyset$, that is, 
that there exists a $v \in S$ such that 

(x) $v \leq b$ and 

(xi) $v \leq c$. Then, using (ix) and (xi), it follows 

(xii) $v < b$ and, by (iv), (x), and (xi), it follows 

(xiii) $v \leq y$. Also, by (viii) and (x), we have 

(xiv) $v < a \vee v$ and, by (v), (vi), and (xiii), we get 

(xv) $a \vee v < x \vee a \vee v$. 
Using (i), (ii), (iii), (ix)-(xi), we also have 

(xvi) $x \vee a \vee v < a \vee b$. Now, by (i), (viii), (ix), and (xi) we have 

(xvii) $a \vee v \parallel b$. Finally, by (i), (ii), (viii), (ix), and (xi), we get 

(xviii) $x \vee a \vee v \parallel b$. 

\noindent So, by (xiv)-(xviii) and (xii), we get  
$v < a \vee v < x \vee a \vee v < a \vee b$, $v < b < a \vee b$, $a \vee v \parallel b$, and
$x \vee a \vee v \parallel b$. 
That is, the set $T_4 = \{ v, a \vee v, x \vee a \vee v, b, a \vee b \}$ (see Figure \ref{T2andT4}) is a sub join semilattice of ${\bf S}$ isomorphic to $M_4$. 
Finally, as $L_{T_4}(a, b) \neq \emptyset$, for all $a, b \in T_4$, it trivially follows that $T_4$ is also semi-strong. 
\end{proof}

\begin{figure} [ht]
\begin{center}

\begin{tikzpicture}

    \tikzstyle{every node}=[draw, circle, fill=white, minimum size=4pt, inner sep=0pt, label distance=1mm]


    \draw (1,.5)	node (b)	[label=below: $b$]		{};
    \draw (0,2)		node (ab)	[label=left: $a \vee b$]{};
    \draw (-1,1)	node (ax)	[label=left: $a \vee x$]{};
	\draw (-1,0)	node (a)	[label=left: $a$]		{};
      
        
	\draw (a)--(ax)--(ab)--(b);  
 

    \draw (6,.5)	node (b4)	[label=right: $b$]				{};
    \draw (5,2)		node (ab4)	[label=right: $a \vee b$]		{};
    \draw (4,1)		node (xav)	[label=left: $x \vee a \vee v$]	{};
	\draw (4,0)		node (av)	[label=left: $a \vee v$]		{};
    \draw (5,-1)	node (v)	[label=right: $v$]				{};
      
        
	\draw (v)--(av)--(xav)--(ab4)--(b4)--(v);  
 
\end{tikzpicture}

\end{center}
\caption{\label{T2andT4} Sub join semilattices $T_2$ and $T_4$}
\end{figure}
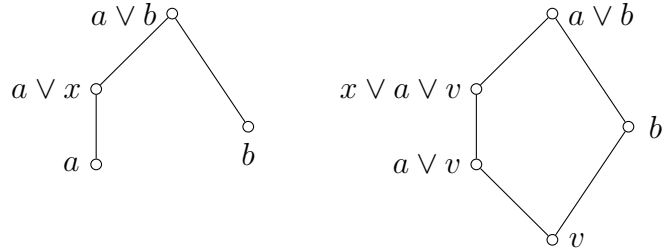

That said, the observant reader may have noticed that Theorem \ref{T1} together with Theorem \ref{T2} do not constitute a characterization (meaning a biconditional) of $o$-modular join semilattices \emph{simpliciter}, due to the fact that Theorem \ref{T2} is not the reverse implication of Theorem \ref{T1}. 
However, the end of the proof of Theorem 2 suggests that we can get a characterization in the case of \emph{complete} (so also for \emph{finite}) $o$-modular join semilattices. 
Indeed, we have the following results. 

\begin{cor}
A \emph{complete} join semilattice is not $o$-modular if and only if it contains either a semi-strong sub semilattice isomorphic to $M_2$ or a strong sub semilattice isomorphic to $M_4$.   
\end{cor}

\begin{proof}
One direction was already taken care of when proving Theorem \ref{T1}. 
For the other direction, proceed as in the case of the proof of Theorem \ref{T2}. 
Now note that at the end of the proof of Theorem \ref{T2} we only concluded $T_4$ to be \emph{semi-strong}. 
This is because it could exist an element $n \in S$ such that $n \leq b$, $n \leq x \vee a \vee v$, and $n \nleq a \vee v$.   
In this case, if the given join semilattice is complete, 
there exists $\bigvee n = \bigvee \{n \in S: n \leq b$, $n \leq x \vee a \vee v \}$.   
Now let us see that the sub join semilattice $T_5 = \{ \bigvee n, a \vee \bigvee n, x \vee a \vee v, b, a \vee b \}$ is, again, a sub join semilattice of ${\bf S}$ isomorphic to $M_4$ (see Figure \ref{T5}). 
It should be clear that both $\bigvee n \leq b$ and $\bigvee n \leq x \vee a \vee v$, 
the last of which implies, by parts (i), (ii), and (xi) of the proof of Theorem \ref{T2}, that $\bigvee n \leq c$. 
So, by part (iv), $\bigvee n \leq y$. 
Moreover, $\bigvee n < b$, as if  $\bigvee n = b$, then $b \leq c$, in contradiction with (ix). 
Also, $\bigvee n < a \vee \bigvee n$, 
because if $\bigvee n = a \vee \bigvee n$, then $a \leq \bigvee n$ and so $a \leq b$, contradicting (viii).   
Furthermore, $a \vee \bigvee n < x \vee a \vee v$, as if $a \vee \bigvee n = x \vee a \vee v$, then $x \leq a \vee \bigvee n$ and so $x \leq y$, in contradiction with (vi). 
Finally, $b \parallel a \vee \bigvee n$, by (i), (viii), and (ix).  
Taking into account (xvi) and (xviii), 
we have that $\bigvee n < a \vee \bigvee n < x \vee a \vee v < a \vee b$, $\bigvee n < b < a \vee b$, $a \vee \bigvee n \parallel b$, and $x \vee a \vee v \parallel b$.
\noindent Furthermore, note that $T_5$ is strong: 
take a $z \in S$ such that $z \leq b$ and $z \leq x \vee a \vee v$. 
Then $z \in \{n \in S: n \leq b$, $n \leq x \vee a \vee v \}$ and so, $z \leq \bigvee n$.    
\end{proof}

\begin{figure} [ht]
\begin{center}

\begin{tikzpicture}

    \tikzstyle{every node}=[draw, circle, fill=white, minimum size=4pt, inner sep=0pt, label distance=1mm]


    \draw (2,-1)	node (b)	[label=right: $b$]				{};
    \draw (0,2)		node (ab)	[label=right: $a \vee b$]		{};
    \draw (-2,0)	node (xav)	[label=left: $x \vee a \vee v$]	{};
	\draw (-2,-2)	node (av)	[label=left: $a \vee v$]		{};
    \draw (0,-4)	node (v)	[label=right: $v$]				{};
    
    \draw (-1,-1)	node (an)	[label=right: $a \vee \bigvee n$]	{};
    \draw (0,-2)	node (n)	[label=below: $\bigvee n$]		{};
      
        
	\draw (v)--(av)--(xav)--(ab)--(b)--(v);  
	\draw (n)--(an)--(xav);
	\draw (n)--(b);
 
\end{tikzpicture}

\end{center}
\caption{\label{T5} Finding $T_5$ inside $T_4$}
\end{figure}
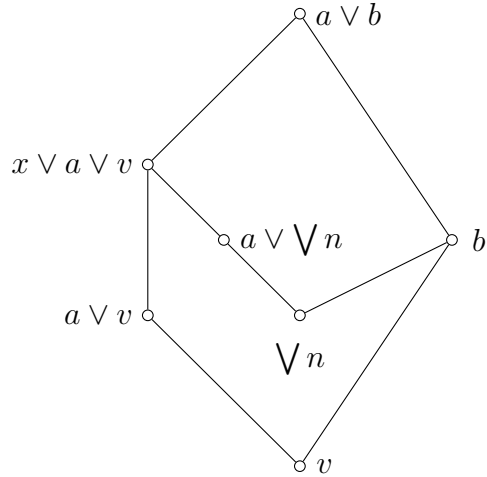

\begin{cor}
A \emph{finite} join semilattice is not $o$-modular if and only if it contains either a semi-strong sub semilattice isomorphic to $M_2$ or a strong sub semilattice isomorphic to $M_4$.   
\end{cor}

\end{document}